\documentclass{article}
\usepackage{array,epsfig,amssymb,amsfonts,bbm,amsmath,epstopdf,algorithm,algorithmic,subfigure,cite,tabularx}
\usepackage{fancyhdr}
\usepackage{latexsym}

\newcolumntype{L}[1]{>{\raggedright\let\newline\\\arraybackslash\hspace{0pt}}m{#1}}
\newcolumntype{C}[1]{>{\centering\let\newline\\\arraybackslash\hspace{0pt}}m{#1}}
\newcolumntype{R}[1]{>{\raggedleft\let\newline\\\arraybackslash\hspace{0pt}}m{#1}}
\newtheorem{theorem}{\textbf{Theorem}}
\newtheorem{lemma}{\textbf{Lemma}}
\newtheorem{proposition}{Proposition}
\newtheorem{corollary}{Corollary}

\newenvironment{proof}[1][Proof]{\begin{trivlist}
\item[\hskip \labelsep {\bfseries #1}]}{\end{trivlist}}

\newcommand{\diag}[1]{$\text{diag}\left(#1\right)$}

\makeatletter
\renewcommand\paragraph{\@startsection{paragraph}{4}{\z@}%
            {-2.5ex\@plus -1ex \@minus -.25ex}%
            {1.25ex \@plus .25ex}%
            {\normalfont\normalsize\bfseries}}
\makeatother
\begin{document}

\title{Restricted Low-Rank Approximation via ADMM}

\author{Ying Zhang \\zy013@ie.cuhk.edu.hk}


\maketitle

\begin{abstract} 
\large{
The matrix low-rank approximation problem with additional convex constraints can find many applications and has been extensively studied before. However, this problem is shown to be nonconvex and NP-hard; most of the existing solutions are heuristic and application-dependent. In this paper, we show that, other than tons of application in current literature, this problem can be used to recover a feasible solution for SDP relaxation. By some sophisticated tricks, it can be equivalently posed in an appropriate form for the Alternating Direction Method of Multipliers (ADMM) to solve. The two updates of ADMM include the basic matrix low-rank approximation and projection onto a convex set. Different from the general non-convex problems, the sub-problems in each step of ADMM can be solved exactly and efficiently in spite of their non-convexity. Moreover, the algorithm will converge exponentially under proper conditions. The simulation results confirm its superiority over existing solutions. We believe that the results in this paper provide a useful tool for this important problem and will help to extend the application of ADMM to the non-convex regime.}
\end{abstract}
\newpage
\subsection*{Notations and Operators}

Vectors and matrices are denoted by boldface lower and upper case letters respectively. The set of real and natural numbers are represented by $\mathbb{R}$ and $\mathbb{N}$. The set of vectors and matrices with proper sizes are denoted as $\mathbb{R}^n,\mathbb{R}^{m\times n}$. The vectors are by default column vectors and the $i^{\text{th}}$ entry of vector $\mathbf{x}$ is denoted as $\mathbf{x}_i$. The entry of matrix $\mathbf{A}$ in the $i^{\text{th}}$ row and $j^{\text{th}}$ column is denoted as $\mathbf{A}_{i,j}$. The superscript $\left(\cdot\right)^T$ stands for transpose.

$\|\cdot\|_p$ denotes the $p$-norm of a vector ($p\geq 1$), \textit{i.e.}, $$\|\mathbf{x}\|_p = \left(\sum_{i}|\mathbf{x}_i|^p\right)^{\frac{1}{p}}$$ and $\|\cdot\|_F$ denotes the $F-$norm of a matrix, \textit{i.e.},
$$\|\mathbf{A}\|_F = \left(\sum_{i,j}A_{i,j}^2\right)^{\frac{1}{2}}.$$

To be consistent with Matlab operator, \diag{\mathbf{A}} returns a vector $\bf a$ with $\mathbf{a}_i = \mathbf{A}_{i,i}$ if the input $\mathbf{A}$ is a matrix and \diag{\mathbf{a}} returns a matrix $\mathbf{A}$ with $\mathbf{A}_{i,j} = \begin{cases} \mathbf{a}_i \quad \text{if } i == j,\\0,\quad \text{otherwise}\end{cases}$. $\text{rank}(\mathbf{X})$ returns the rank of matrix $\mathbf{X}$. 
\newpage
\section{Introduction}
There is a common belief that the complexity of the systems we study is up to a limited level and the useful information in our observation is usually sparse. The property of sparsity can be applied to reduce the necessary sampling rate of signal reconstruction in compressed sensing \cite{donoho2006compressed}, to avoid over-fitting in regularized machine learning algorithms \cite{scholkopf2002learning}, to increase robustness against noise in solving linear equations, to train the deep neural network in unsupervised learning \cite{lee2009convolutional} and beyond.

A direct interpretation of sparsity is that there are many zeros entries in the data.
More often, sparsity is reflected not directly by the data we collect, but by the underlying pattern to be discovered. For example, the sparsity of the voice will be more prominent in frequency domain \cite{bofill2001underdetermined}; the sparse pattern of some astronomical or biomedical imaging data will be reviewed by wavelet transformation \cite{starck2010astronomical,unserwavelets}. In recent years, we are able to generate or collect a lot of data samples with many features, and different data samples and features are usually correlated with each other. If we form the data into a matrix, the matrix is often of low rank, which is another representation of sparsity and can be exploited in many applications \cite{fazel2002matrix}. To encourage sparsity, we can incorporate the sparsity interpretation into the objective function as a regularization term, or put a sparsity upper bound as a hard constraint.

The sparsity usually leads to a non-convex problem and the problem is intractable. For example, we can let the number of nonzero component (0-norm) in the data to be upper bounded. However, the function of 0-norm is non-convex.
To tackle the non-convexity, the researchers proposed many solutions, usually belonging to the following two categories.
First, we can use its convex envelop to replace the non-convex function, for example, to use 1-norm to replace 0-norm, and solve the convex problem instead (this method is usually called convex relaxation). And we can study under what condition the convex relaxation is exact or how much performance loss we will suffer due to the relaxation \cite{donoho2006compressed}. Second, we can directly design some heuristic algorithms to solve the non-convex problem based on engineering intuition, like the Orthogonal Matching Pursuit (OMP) algorithm in compressed sensing \cite{needell2009cosamp}. This kind of method usually works well in practice but little theoretical performance guarantee can be obtained.

\subsection{Related work}

The rank of a matrix, as the dimensionality of the smallest subspace to which the data belongs to, is essentially important and the research on it can be traced back to the origin of matrix. In recent years, many problems with low-rank matrix are of our interests. However, its non-convex nature makes the problems generally intractable. In this case, some researchers proposed to use the nuclear norm or log-det function of the matrix to replace the rank function and then solve the relaxed convex problem \cite{recht2010guaranteed,fazel2003log}. In \cite{berman2003completely}, the matrix being of low rank is interpreted as the matrix being factorized into two smaller-sized matrices, then the matrix factorization techniques can be applied.

The alternating direction method of multiplier (ADMM) was invented in the 1970s and is raising its popularity in the era of big data; many large scale optimization problems that arise in practice can be formulated or equivalently posed in a form appropriate for ADMM and distributed algorithms can be obtained \cite{boyd2011distributed}. Many successful applications including consensus and sharing have been found. The theory related to the convergence rate of ADMM is an active ongoing topic and many results for the convex problem are provided in \cite{boyd2011distributed} and the references therein. In practice, we can meet many problems that are non-convex and ADMM algorithm is proposed to tackle them in a heuristic sense, such as optimal power flow problem \cite{you2014non}, matrix factorization \cite{chartrand2012nonconvex}, etc. However, in the non-convex regime, little theoretical guarantee can be obtained. To the best of our knowledge, \cite{magnusson2014convergence, hongconvergence} are among the first attempts to characterize the convergence behavior of ADMM for the general non-convex problems and has limited theoretical performance guarantee.

In this paper, we consider that low-rank approximation problem with additional convex constraints, which is non-convex and NP-hard. By borrowing some tricks from \cite{boyd2011distributed} \cite{you2014non}, we can reformulate the problem into the form that can be solved by ADMM. In each iteration of ADMM, the first update is to solve a convex problem and the second update is to solve a standard low-rank approximation, the optimal solution of which can be obtained by singular value decomposition. Different from the general non-convex problems, the sub-problems in ADMM can be solved efficiently. This approach is motivated by \cite{you2014non}, in which the authors also use ADMM to solve a non-convex problem with similar structures.  
The different thing is that, due to the special objective function we consider, \textit{i.e.}, $\|X-\hat{X}\|_F$, we can show that the first update is a projection onto a convex set and then establish the convergence of primal variable by assuming the convergence of dual variable, which is an appealing result for the ADMM application in non-convex regime.

The remaining part of this paper is organized as follows,

$\vartriangleright$ In Section~\ref{sec:problemformulation}, we introduce the low rank approximation problem with additional convex constraints and show that it can capture the structured low rank approximation and feasible solution recovery of SDP relaxation as two special cases.

$\vartriangleright$ In Section~\ref{sec:algorithm}, we leverage ADMM to design a generic algorithm and shows that each step of this algorithm can be solved efficiently. We also provide some theoretical results to characterize its convergence behavior.

$\vartriangleright$ In Section~\ref{sec:simulation}, we shows the performance of our algorithm by extensive evaluations with synthetic and real-world data.

$\vartriangleright$ Section~\ref{sec:conclusion} is for conclusion and future work.

\section{Restricted Low Rank Approximation}\label{sec:problemformulation}

\subsection{Problem formulation}

In this paper, we are particularly interested in the data-fitting problem and restrict our attention to the low-rank solutions. More precisely, we are given a matrix $\hat{\mathbf{X}}$ and we want to find a low-rank matrix to approximate $\hat{\mathbf{X}}$. The first version of this problem is formulated as the Low Rank Approximation problem (\textbf{LRA}) as follows,
\begin{subequations}
\begin{eqnarray}\label{prob:LRA}
\textbf{LRA}\quad \min_{X}&&\|\hat{X}-X\|_F
\nonumber\\
\textrm{s.t.}
& & \text{rank}(X)\leq K;\\
\text{var.} &&X,\nonumber
\end{eqnarray}
\end{subequations}
where $K$ is an integer to specify the upper bound of the matrix rank.
The problem is non-convex due to the rank constraint but an optimal solution can be given by the well known Eckart-Young-Mirsky Theorem,

\begin{theorem}[Eckart-Young-Mirsky Theorem \cite{eckart1936approximation}]\label{theorem:EYM}
If the matrix $\hat{X}$ admits the singular value decomposition $\hat{X} = U\Sigma V^H$ with
$\Sigma = \text{diag}([\sigma_1,\sigma_2,\cdots,\sigma_n])$ and $\sigma_1\geq\sigma_2\geq \cdots \geq \sigma_n \geq0$, an optimal solution to problem \textbf{LRA} is given by
$$X^* = \sum_{k=1}^K\sigma_k u_kv^H_k.$$
Furthermore, the minimizer is unique if $\sigma_K$ and $\sigma_{K+1}$ are not equal.
\end{theorem}

This method is called truncated singular value decomposition (SVD). We can see that even though \textbf{LRA} is non-convex, it can be solved in polynomial time because SVD can be computed in polynomial time. 

In this paper, we put some additional constraints on \textbf{LRA} and call the new problem as Restricted Low Rank Approximation problem (\textbf{RLRA}). The problem is casted as follows,
\begin{subequations}
\begin{eqnarray}\label{prob:RLRA}
\textbf{RLRA}\quad \min_{X}&&\|\hat{X}-X\|_F^2
\nonumber\\
\textrm{s.t.}
& & \text{rank}(X)\leq K;\\
&　& g(X)\leq 0;\\
\text{var.}
& & X,\nonumber
\end{eqnarray}
\end{subequations}
where $g(X)$ is a convex function, requiring that the approximation $X$ is located in a convex set.

The difference between \textbf{LRA} and \textbf{RLRA} is the constraint $g(\mathbf{X})\leq 0$. The convexity of $g(\mathbf{X})$ will makes the readers feel that the difference is not significant, since the convex things are usually treated as easy in most literature. However, this is not true. The new constraint makes the truncated SVD not applicable to solve \textbf{RLRA} \footnote{It is easy to imagine that the solution by truncated SVD may not respect $g(\mathbf{X})\leq 0$.}. More importantly, \textbf{RLRA} is believed to be NP-hard \cite{markovsky2008structured}. There is no hope to always achieve the optimal solution in polynomial time unless P=NP.

There are several proper ways to tackle this non-convex problem. Firstly, we can replace the rank constraints by some constraints on the nuclear norm of $\mathbf{X}$ \cite{fazel2002matrix} and solve the relaxed problem instead; secondly, we can use the kernel representation, for example,to equivalently transform the rank constraint to the fact that the matrix can be factorized into two smaller sized matrices and then apply the matrix factorization technique \cite{markovsky2008structured}. In this paper, we propose to use ADMM, which has found its success in many problems \cite{boyd2011distributed}, to solve this problem.


\subsection{Two specific instances}\label{sec:twocases}

The generic problem $\textbf{RLRA}$ can be used to solve many problems and we review two of them in this section to provide more motivations.

\subsubsection{Low-rank approximation with linear structures}
The data fitting task can be formulated as a structured low-rank approximation problem (\textbf{SLRA})if the system generating the data is of linear model and bounded complexity \cite{markovsky2008structured}. Many applications can be found in system theory, signal processing, computer algebra, etc \cite{markovsky2008structured}.

We denote $\mathcal{A}$ a set of matrices with specific \textit{affine} structures, for example, Hankel matrix, Toplitiz matrix, \textit{etc}. The problem is formally given as follows,
\begin{subequations}
\begin{eqnarray}\label{prob:SLRA}
\textbf{SLRA} \min_{x}&&\|X-\hat{X}\|_F^2
\nonumber\\
\textrm{s.t.}
& & \text{rank}(X) \leq K, \label{eq:rank}\\
& & X\in \mathcal{A}, \\
\text{var.}
& & X\in\mathbb{C}^{m\times n}.\nonumber
\end{eqnarray}
\end{subequations}

Different applications lead to different requirements of $\mathcal{A}$ \cite{markovsky2008structured}. Without diving into to the details, we list some of them here,
\begin{itemize}
\item Hankel matrix: approximate realization, model reduction, output error identification;
\item Sylvester matrix: Pole place by low-order controller, approximate common divisor;
\item Hankel\&Toeplitz matrix: Harmonic retrieval;
\item Non-negative matrix: image mining, Markov chains.
\end{itemize}

Some heuristic or local-optimization based algorithms for different problems are summarized in \cite{markovsky2008structured}.

\subsubsection{Feasible solution recovery of SDP relaxation}

Many communication problems, like multicast downlink transmit beamforming problem, can be formulated as a  quadratically constrained quadratic program \textbf{QCQP}, which is non-convex and generally NP-hard,
\begin{subequations}
\begin{eqnarray}\label{prob:QCQP}
\textbf{QCQP}\quad \min_{x}&&x^HCx
\nonumber\\
\textrm{s.t.}
& & x^HF_ix\geq g_i, i = 1,...,k \nonumber\\
& & x^HH_ix = l_i, i = 1,...,m \nonumber\\
& &\text{var.} x\in \mathbb{R}^n,\nonumber
\end{eqnarray}
\end{subequations}
and it can be shown to be equivalent to problem \textbf{SDP-QCQP}.

\begin{subequations}
\begin{eqnarray}\label{prob:QCQP}
\textbf{QCQP-SDP}\quad \min_{x}&&\text{trace}(C\cdot X)
\nonumber\\
\textrm{s.t.}
& & \text{trace}(F_i\cdot X)\geq g_i, i = 1,...,k \label{eq:QCQP10}\\
& & \text{trace}(H_i\cdot X) = l_i, i = 1,...,m \label{eq:QCQP20}\\
& & \text{rank}(X) = 1, \label{eq:rank0}\\
\text{var.}
& & X\succeq 0, X\in\mathbb{C}^{n\times n}.\nonumber
\end{eqnarray}
\end{subequations}
By dropping the rank-1 constraint \eqref{eq:rank0} we can have a standard SDP problem (denoted as \textbf{QCQP-SDPR}) and it can be solved by standard solver like CVX \cite{grant2008cvx}. This technique is called SDP relaxation and more details can be found in \cite{luo2010semidefinite}.

If the optimal solution of the relaxed problem, denoted as $\hat{X}$, happens to respect the rank-1 constraint, then the optimal solution of the problem \textbf{QCQP}, denoted as $x^*$, can be obtained by the fact that $\hat{X} = x^*x^{*T}$. In this case, the SDP relaxation is called exact relaxation, and the original \textbf{QCQP} problem can be solved efficiently and exactly even though it is non-convex. For the optimal power flow problem, the exact relaxation always happens if some conditions hold \cite{low2013convex} \cite{lavaei2012zero}.

However, the rank of $\hat{X}$ is more often larger than 1 \footnote{Otherwise SDP relaxation will always solve a non-convex problem exactly, which is not true} and the relaxation is not exact. In this case, another step is needed to recover a \textit{good} solution if we do not want to waste pervious effort. The most direct way is to find the rank-1 matrix that is closest to $\hat{\mathbf{X}}$, \textit{i.e.}, solving \textbf{LRA} with $K=1$ by truncated SVD. This approach is suggested in \cite{zhu2011estimating} for the state estimation problem of power system. However, we want to point out that this approach is not guaranteed to produce a feasible solution of the original problem, because the rank-1 matrix by truncated SVD may not satisfy the other constraints of \textbf{SDP-QCQP} like \eqref{eq:QCQP10} \eqref{eq:QCQP20}. We provide a successful case in Fig~\ref{fig:works} and an unsuccessful case in Fig~\ref{fig:fails} for a more clear illustration; $\tilde{\mathbf{X}}$ is the optimal solution of \textbf{QCQP-SDPR} and $\mathbf{X}_1$ is the rank-1 approximation of $\tilde{\mathbf{X}}$.

\begin{figure}
\begin{minipage}{0.45\columnwidth}
\includegraphics[width=0.9\columnwidth]{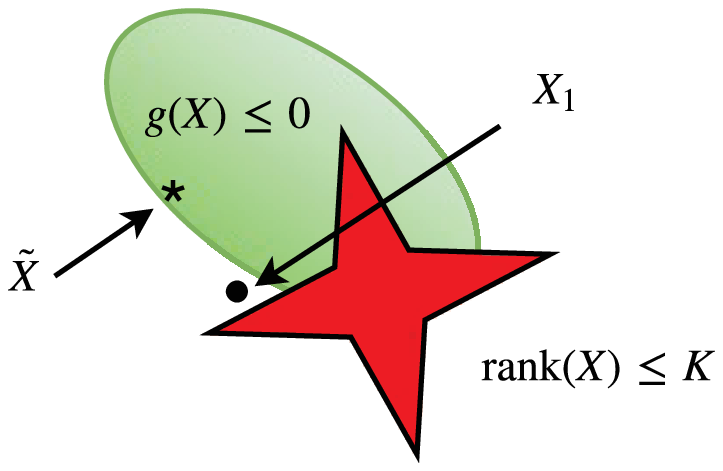}\\
\caption{An example that truncated SVD works for \textbf{FSR-SDPR}}\label{fig:works}
\end{minipage}
\begin{minipage}{0.45\columnwidth}
\includegraphics[width =0.9\columnwidth]{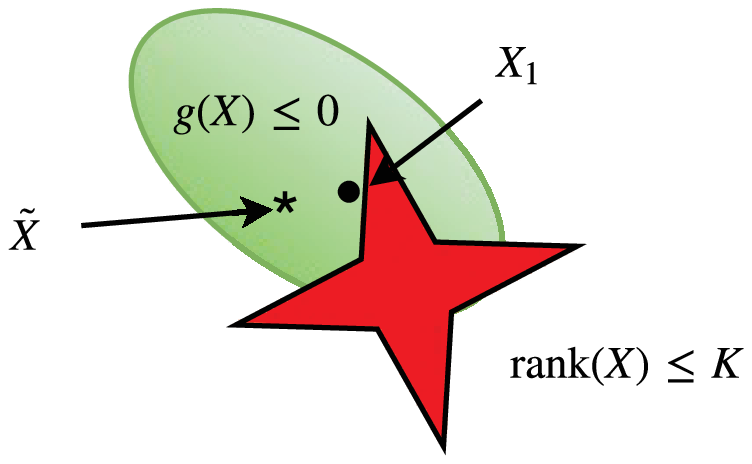}\\
\caption{An example that truncated SVD fails to work for \textbf{FSR-SDPR}}\label{fig:fails}
\end{minipage}
\end{figure}

Some heuristic and problem-dependent algorithms are proposed in \cite{luo2010semidefinite}. For example, if it is required that $\|\mathbf{x}\|_2\leq 1$ for \textbf{QCQP}, we can first obtain a solution $\hat{\mathbf{x}}$ from the rank-1 matrix by truncated SVD and then normalize the solution by $\frac{\hat{\mathbf{x}}}{\|\hat{\mathbf{x}}\|_2}$ if $\|\hat{\mathbf{x}}\|_2 > 1$.

Here we propose a generic method to solve the problem \textbf{FSR-SDPR} (Feasible Solution Recovery of SDP Relaxation) by finding a rank-1 matrix from which we can recover a feasible solution to \textbf{QCQP}.
\begin{subequations}
\begin{eqnarray}\label{prob:QCQP}
\textbf{FSR-SDPR}\quad \min_{x}&&\|X-\hat{X}\|_F
\nonumber\\
\textrm{s.t.}
& & g(\mathbf{X})\leq 0,\nonumber\\
& & \text{rank}(X) = 1, \nonumber\\
\text{var.}
& & X\succeq 0, X\in\mathbb{C}^{n\times n},\nonumber
\end{eqnarray}
\end{subequations}
where $g(\mathbf{X})\leq 0$ represents the original convex constraints like \eqref{eq:QCQP10} and \eqref{eq:QCQP20}.

The rationale behind this approach is that we want to find the point closest to $\hat{X}$ in the feasible region. If the objective function, $\text{trace}(C\cdot X)$, is Lipschitz continuous, the performance of the recovered solution is guaranteed to be close to the optimal value of \textbf{QCQP-SDPR}, hence close to the optimal solution of \textbf{QCQP}. It is not difficult to see that solving this problem can be viewed as a special case of \textbf{RLRA} and the solution is far from being trivial.

\section{Algorithm Design}\label{sec:algorithm}
Next we will present how to solve the problem \textbf{RLRA} via the ADMM algorithm.

\subsection{ADMM algorithm}
In this section, we review the basic version of ADMM \cite{boyd2011distributed} to bring all the readers to the same page \footnote{The readers are recommended to read \cite{boyd2011distributed} for more details.}. We present the standard problem that ADMM can solve in \textbf{SP}, in which the objective function is separable and two variables $x,y$ are coupled with each other by a linear constraint.

\begin{subequations}
\begin{eqnarray}\label{prob:ADMM}
\textbf{SP}\quad \min_{x,y}&&f(x)+g(y)
\nonumber\\
\textrm{s.t.}
& & Ax+By = c,\\
\text{var.}
& & x,y.\nonumber
\end{eqnarray}
\end{subequations}

The augmented Lagrange multiplier function of the above problem is given by
\begin{equation*}
L_{\rho}(x,y,\lambda) = f(x)+g(y)+\lambda (Ax+By-c) + \frac{\rho}{2}\|Ax+By-c\|_2^2,
\end{equation*}
the scaled form of which is
\begin{equation}\label{eq:aLagrange}
L_{\rho}(x,y,u) = f(x)+g(y)+\frac{\rho}{2}\|Ax+By-c+u\|_2^2,
\end{equation}
and we will use the scaled form in the sequel unless specified.

In each iteration, ADMM algorithm consists of the following three steps

\begin{itemize}
\item $x$ \textbf{update}:
$$x^{k+1} = \text{argmin}_{x}L_{\rho}(x,y^{k},u^{k}).$$
\item $y$ \textbf{update}:
$$y^{k+1} = \text{argmin}_{y}L_{\rho}(x^{k+1},y,u^{k}).$$
\item $u$ \textbf{update}:
$$u^{k+1} = u^k+Ax^{k+1}+By^{k+1}-c.$$
\end{itemize}

The ADMM algorithm is guaranteed to converge under some conditions, as shown in Theorem~\ref{theorem:ADMM_convex}.

\begin{theorem}[\cite{boyd2011distributed}]\label{theorem:ADMM_convex}
If $f(x)$ and $g(y)$ are closed, proper , convex and the unaugmented Lagrange function $L_0$ has a saddle point, ADMM algorithm is guaranteed to converge in the following sense,
\begin{itemize}
\item The residual $Ax+By-c$ will converge to 0, \textit{i.e.}, the solution $x^k,y^k$ will approach feasibility.
\item The objective value $f(x^k)+g(y^k)$ will converge to the optimality.
\item The dual variable $u$ will converge to the dual optimal point.
\end{itemize}
\end{theorem}

We highlight that under the conditions provided in Theorem~\ref{theorem:ADMM_convex}, the primal variables $x,y$ are not guaranteed to converge. But since the objective value is guaranteed to converge, all the solutions after enough iterations will produce the same results and are equivalently good.

\subsubsection{Some remarks}
We remark some properties of ADMM here.
First, the ADMM algorithm can produce a reasonably accurate solution very fast, but takes more time to generate a high accurate solution. For many machine learning or statistical tasks, the overall performance depends on both problem formulation (including feature engineering and model selection) and parameter estimation (solving the optimization problem). Usually the performance bottleneck comes from the first one and a more accurate solution will not lead to significant performance improvement. In this way, a reasonably accurate solution is totally acceptable and this property of ADMM is appealing. Second, the optimization variable of Problem \textbf{SP} is $\left(x,y\right)$. In conventional algorithm such as gradient descent, $x,y$ are updated simultaneously, but in ADMM, we update $x$ first and use the newly updated $x$ to update $y$. The intuition is that we want the newly updated information to take effect as soon as possible, which is similar to the logic of Gauss-Sidel algorithm. Last but not least, the ADMM algorithm can be implemented in a distributed manner with a parameter server \cite{boyd2011distributed}. 
\subsection{Problem reformulation and applying ADMM}
The original problem formulation in \textbf{RLRA} does not have the form of \textbf{SP} to be readily solved by ADMM. We need to reformulate the problem. Firstly we define two indicator functions as follows,
\begin{equation}
\mathcal{I}(X) = \begin{cases}0,\quad\text{if }\text{rank}(X) \leq K,\\ +\infty,\quad\text{otherwise},\end{cases}
\end{equation}
\begin{equation}
\mathcal{J}(X) = \begin{cases}0,\quad\text{if }g(X) \leq 0,\\ +\infty,\quad\text{otherwise}.\end{cases}
\end{equation}
And we reformulate Problem \textbf{RLRA} equivalent to Problem \textbf{RLRA-ADMM} as,
\begin{subequations}
\begin{eqnarray}\label{prob:RLRA-ADMM}
\textbf{RLRA-ADMM} \quad\min_{X,Y}&&\underbrace{\|X-\hat{X}\|_F+\mathcal{J}(X)}_{X\text{ involved}}+\underbrace{\mathcal{I}(Y)}_{Y\text{ involved}}
\nonumber\\
\textrm{s.t.}
& & X-Y = 0\\
\text{var.}
& & X,Y.\nonumber
\end{eqnarray}
\end{subequations}

The equivalence of the two problems is formally established in Lemma~\ref{lemma:equivalence}.

\begin{lemma}\label{lemma:equivalence}
If Problem \textbf{RLRA} is feasible, then the optimal solutions and optimal objective values of \textbf{RLRA} and \textbf{RLRA-ADMM} are the same.
\end{lemma}
And more importantly, \textbf{RLRA-ADMM} is readily solved by AMDD algorithm. We denote $f(X) = \|X-\hat{X}\|_F+\mathcal{I}(X)$, $g(Y) = J(Y)$ and the augmented function with dual variable $U$ as
$$L_{\rho}(X,Y,U) = f(X)+g(Y)+\frac{\rho}{2}\|X-Y+U\|_F^2.$$
Following the ADMM procedure, we will have the following three updates in each iteration.
\begin{itemize}
\item $X$ \textbf{update}:
$$X^{k+1} = \text{argmin}_{X}L_{\rho}(X,Y^{k},U^{k}).$$
\item $Y$ \textbf{update}:
$$Y^{k+1} = \text{argmin}_{Y}L_{\rho}(X^{k+1},Y,U^{k}).$$
\item $U$ \textbf{update}:
$$U^{k+1} = U^k+X^{k+1}-Y^{k+1}.$$
\end{itemize}

In the next part, we carefully study the details of the algorithm and show that each update can be carried out efficiently even though the problems can be non-convex.
\subsubsection{The subproblems}
The $U$ update is simple and direct. We revisit the other two in this part.
\paragraph{$X$ update}
The optimization problem in $X$ update is equivalent to
\begin{subequations}
\begin{eqnarray*}\label{prob:ADMM}
\textbf{X-MIN}\min_{X}&& \|X-\hat{X}\|_F^2+\frac{\rho}{2}\|X-Y^k+U^k\|_F^2
\nonumber\\
\textrm{s.t.}
& & g(X)\leq 0.\nonumber
\end{eqnarray*}
\end{subequations}

The problem \textbf{X-MIN} is convex and can be solved efficiently. We wit can be viewed as a projection of $\mathbf{X}^{k+1}+\mathbf{U}^k$ onto the convex set $\mathcal{S} = \{\mathbf{Y}|g(\mathbf{Y})\leq 0\}$. It can be solved efficiently or even have closed-form solutions.

\begin{lemma}
The optimal solution of \textbf{X-MIN} can be obtained by solving
\begin{subequations}
\begin{eqnarray*}\label{prob:ADMM}
\min_{X}&& \|X-\frac{1}{1+\frac{\rho}{2}}\left(\hat{X}+\frac{\rho}{2}(Y^k-U^k)\right)\|_F^2
\nonumber\\
\textrm{s.t.}
& & g(X)\leq 0.\nonumber
\end{eqnarray*}
\end{subequations}
which is a projection of $\frac{1}{1+\frac{\rho}{2}}\left(\hat{X}+\frac{\rho}{2}(Y^k-U^k)\right)$ onto the convex set $\{X|g(X)\leq 0\}$.
\end{lemma}
\begin{proof}
We prove the equivalence between the two problems by showing the linear relationship between the two objective functions. We will use the equation that $\|A\|_F^2 = \text{trace}(AA^H)$ in the proof.

\begin{align*}
&\|X-\hat{X}\|_F^2+\frac{\rho}{2}\|X-Y^k+U^k\|_F^2\\
=& \text{trace}((X-\hat{X})(X-\hat{X})^H)+\frac{\rho}{2}\text{trace}((X-Y^k+U^k)(X-Y^k+U^k)^H)\\
=& \text{trace}\left(\left(1+\frac{\rho}{2}\right)XX^H-2\left(\hat{X}+\frac{\rho}{2}(Y^k-U^k)\right)X^H+C_1\right)\\
=& \left(1+\frac{\rho}{2}\right)\text{trace}\left(\left(X-\frac{1}{1+\frac{\rho}{2}}\left(\hat{X}+\frac{\rho}{2}(Y^k-U^k)\right)\right)
\left(X-\frac{1}{1+\frac{\rho}{2}}\left(\hat{X}+\frac{\rho}{2}(Y^k-U^k)\right)\right)^H\right)\\
&\quad+C_2\\
=& \left(1+\frac{\rho}{2}\right)\|X-\frac{1}{1+\frac{\rho}{2}}\left(\hat{X}+\frac{\rho}{2}(Y^k-U^k)\right)\|_F^2 + C_2.
\end{align*}
The proof is completed.
\end{proof}

\paragraph{$Y$ update}

The optimization problem in $\mathbf{Y}$ update is equivalent to
\begin{subequations}
\begin{eqnarray*}\label{prob:ADMM_Y}
\textbf{Y-MIN}\quad \min_{\mathbf{Y}}&& \frac{\rho}{2}\|\mathbf{X}^{k+1}-\mathbf{Y}+\mathbf{U}^k\|_F^2
\nonumber\\
\textrm{s.t.}
& & \text{rank}(\mathbf{Y})\leq 0.\nonumber
\end{eqnarray*}
\end{subequations}

The above optimization problem is non-convex because the rank constraint is non-convex. Thus the problem is challenging on the first sight. The following lemma shows that its optimal solution can be obtained by truncated SVD.
We summarize the algorithm into Algorithm~\ref{alg:admm-rlra} to end this part.

\begin{algorithm}[!ht]
\protect\caption{ADMM-RLRA: ADMM algorithm for \textbf{RLRA}}
\begin{algorithmic}[1]
\REQUIRE $p_{m}$,$p_{g}$,$p_{e}(t)$,$e^{k}(t)$
\ENSURE $u^{k}(t)$,$v^{k}(t)$
\STATE initialization
\WHILE{ not terminate}
\STATE Obtain $\mathbf{X}^{k+1}$ by solving \textbf{X-MIN}
\STATE Obtain $\mathbf{Y}^{k+1}$ by solving \textbf{Y-MIN}
\STATE $\mathbf{U}^{k+1} = \mathbf{U}^k+\mathbf{X}^{k+1}-\mathbf{Y}^{k+1}$
\STATE $k=k+1$
\ENDWHILE
\end{algorithmic}\label{alg:admm-rlra}
\end{algorithm}

\subsection{On the convergence of the algorithm}\label{sec:convergence}

The convergence of ADMM for the non-convex problem is an open problem \cite{boyd2011distributed}. Some positive results are obtained with some assumptions \cite{magnusson2014convergence,hongconvergence}. We provide some preliminary results in this section, regarding to the convergence and feasibility of Algorithm~\ref{alg:admm-rlra}. An assumption that the dual variable $\mathbf{U}^k$ converges is used in the theoretical analysis, which is also the assumption for the convergence analysis of ADMM for polynomial optimization\cite{jiang2014alternating}, non-negative matrix factorization \cite{zhang2010alternating} and non-negative matrix factorization \cite{xu2012alternating}.

\begin{lemma}\label{lemma:feasible}
If the dual variable $\mathbf{U}^k$ converges, the solution $\mathbf{X}^k$ in the $\mathbf{X}$ update of Algorithm~\ref{alg:admm-rlra} will approach feasibility of \textbf{RLRA}.
\end{lemma}
\begin{proof}
Since $\mathbf{U}^k$ converges, $\mathbf{X}^k-\mathbf{Y}^k$ converges to 0. Since $\mathbf{Y}^k$ satisfies the constraint $g(\mathbf{X})\leq 0$ for all $k$, $\mathbf{X}^k$ will satisfy the same condition when $k$ is large enough. Meanwhile, $X^k$ satisfies the rank constraint for all $k$, then $\mathbf{X}^k$ is a feasible solution for \textbf{RLRA}.

The proof is completed.
\end{proof}

When we characterize the convergence of ADMM in Theorem~\ref{theorem:ADMM_convex}, the primal variables are not guaranteed to converge, but can oscillate in the \textit{optimal} region. For the non-convex problem, because we do not know whether the objective function will converge or not, the convergence of the primal variable is more important, which we will discuss next.

Since $\mathbf{U}^k$ converges $\bar{\mathbf{U}}$, $\mathbf{X}^k-\mathbf{Y}^k$ converges to 0, then the $\mathbf{X}$ update can be denoted  as
$$\mathbf{X}^{k+1} = \underset{\text{rank}(\mathbf{X})\leq K}{\text{argmin}}\|\mathbf{X}-\frac{1}{1+\frac{\rho}{2}}\left(\hat{X}+\frac{\rho}{2}(\mathbf{X}^k-\mathbf{U}^k)\right)\|_F^2.$$

Let $\mathbf{D} = \left(\hat{\mathbf{X}}-\frac{\rho}{2}\bar{\mathbf{U}}\right)$, it can be further simplified as
\begin{align*}
\mathbf{X}^{k+1} & = \underset{\text{rank}(\mathbf{X})\leq K}{\text{argmin}}\|\mathbf{X}-\frac{1}{1+\frac{\rho}{2}}\left(\frac{\rho}{2}\mathbf{X}^k+\mathbf{D}\right)\|_F^2\\
& = \underset{g(\mathbf{X})\leq 0}{\text{argmin}}\|\mathbf{X}-\left(\alpha\mathbf{X}^k+(1-\alpha)\mathbf{D}\right)\|_F^2, \quad \alpha = \frac{\rho}{2+\rho}\\
& = \mathcal{C}(\mathbf{X}^k).
\end{align*}

Based on this understanding, we have the following theorem to characterize its convergence.
\begin{theorem}
If the dual variable converges, the primal variable will converge, and when $k$ is large enough, we can have
$$\|\mathbf{X}^{k+2}-\mathbf{X}^{k+1}\|_F\leq \frac{\rho}{\rho+2}\|\mathbf{X}^{k+1}-\mathbf{X}^{k}\|_F.$$
\end{theorem}
\begin{proof}
$X^{k+2}$ is the projection of $\alpha X^{k+1}+(1-\alpha)D$ and $X^{k+1}$ is the projection of $\alpha X^{k}+(1-\alpha)D$, with the fact that the feasible region of the projections is convex, and projection onto a convex set is non-expansive \cite{hiriart2013convex} we can have
\begin{align*}
\|\mathbf{X}^{k+2}-\mathbf{X}^{k+1}\|_F&\leq \|\alpha X^{k+1}+(1-\alpha)D-\alpha X^{k}-(1-\alpha)D\|_F\\
& = \alpha\|X^{k+1}-X^{k}\|_F
\end{align*}
The proof is completed.
\end{proof}
We provide some remarks regarding this theorem. Firstly, the convergence of $X^k$ will guarantee that we can have a local optimal or stationary point. Secondly, the inequality indicates that the primal variable will converge exponentially and a smaller $\rho$ will lead to a faster convergence.


\subsection{Another ADMM}\label{sec:otherorder}
\subsubsection{Algorithm design}
In this part, we provide another ADMM algorithm, which considers different constraints compared with the algorithm we previously proposed, \textit{i.e.}, in X \textbf{update} we consider the rank constraint while in Y \textbf{update} we consider the convex constraint.

We denote $\tilde{f}(X) = \|X-\hat{X}\|_F+\mathcal{I}(X)$, $\tilde{g}(Y) = J(Y)$ and the augmented function with dual variable $U$ as
$$\tilde{L}_{\rho}(X,Y,U) = \tilde{f}(X)+\tilde{g}(Y)+\frac{\rho}{2}\|X-Y+U\|_F^2.$$
Following the ADMM procedure, we will have the following three updates in each iteration.
\begin{itemize}
\item $X$ \textbf{update}:
$$X^{k+1} = \text{argmin}_{X}\tilde{L}_{\rho}(X,Y^{k},U^{k}).$$
\item $Y$ \textbf{update}:
$$Y^{k+1} = \text{argmin}_{Y}\tilde{L}_{\rho}(X^{k+1},Y,U^{k}).$$
\item $U$ \textbf{update}:
$$U^{k+1} = U^k+X^{k+1}-Y^{k+1}.$$
\end{itemize}

\paragraph{The Subproblems}
By similar arguments, the optimization problem in $X$ update is equivalent to
\begin{subequations}
\begin{eqnarray*}\label{prob:ADMM}
\min_{X}&& \|X-\frac{1}{1+\frac{\rho}{2}}\left(\hat{X}+\frac{\rho}{2}(Y^k-U^k)\right)\|_F^2
\nonumber\\
\textrm{s.t.}
& & \text{rank}(X)\leq K.\nonumber
\end{eqnarray*}
\end{subequations}
which is to find a low-rank matrix to approximate $\frac{1}{1+\frac{\rho}{2}}\left(\hat{X}+\frac{\rho}{2}(Y^k-U^k)\right)$ and can be solved by truncated SVD.

The optimization problem in $\mathbf{Y}$ update is equivalent to
\begin{subequations}
\begin{eqnarray*}\label{prob:ADMM_Y}
 \min_{\mathbf{Y}}&& \frac{\rho}{2}\|\mathbf{X}^{k+1}-\mathbf{Y}+\mathbf{U}^k\|_F^2
\nonumber\\
\textrm{s.t.}
& & g(\mathbf{Y})\leq 0.\nonumber
\end{eqnarray*}
\end{subequations}

The problem is convex and it can be viewed as a projection of $\mathbf{X}^{k+1}+\mathbf{U}^k$ onto the convex set $\mathcal{S} = \{\mathbf{Y}|g(\mathbf{Y})\leq 0\}$. It can be solved efficiently or even have closed-form solutions.

\subsubsection{Convergence analysis}
The theoretical analysis is also based on the assumption that the dual variable will converge. With the same argument in Section~\ref{sec:convergence}, the value of primal variable is iteratively updated by the function $Y=\mathcal{H}(X)$, where
$$\mathcal{H}(X) = \underset{\text{rank}(\mathbf{Y})\leq K}{\text{argmin}}\|\mathbf{Y}-\left(\alpha\mathbf{X}+(1-\alpha)\mathbf{D}\right)\|_F^2, \quad \alpha = \frac{\rho}{2+\rho},\footnote{To make function $\mathcal{H}(\mathbf{X})$ well-defined, we assume that the minimizer of the $\mathbf{X}$ update is unique, which means that, in each iteration, $\sigma_K\neq \sigma_{K+1}$ holds for the singular value decomposition of $\left(\alpha\mathbf{X}^k+(1-\alpha)\mathbf{D}\right)$. This assumption is used in the sequel.}$$ \textit{i.e.}, $X^{k+1} = \mathcal{H}(X^k)$ An interpretation of the update is provide in Fig~\ref{fig:X_update}.
\begin{figure}
  \centering
  \includegraphics[width=0.5\columnwidth]{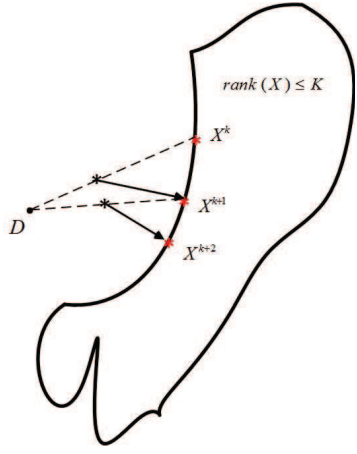}\\
  \caption{An illustration of $\mathbf{X}$ update}\label{fig:X_update}
\end{figure}

If the feasible approximation region is convex, the update rule will converge with $\|\mathbf{X}^{k+2}-\mathbf{X}^{k+1}\|_F\leq \alpha\|\mathbf{X}^{k+1}-\mathbf{X}^{k}\|_F$, but unfortunately, the set $\{X|\text{rank} \leq K\}$ is non-convex and the update is not guaranteed to converge. In the following, we provide some simple necessary conditions for the convergence of the primal variable.
Firstly, if $\mathbf{X}^k$ converges to $\mathbf{X}^*$, then $\mathbf{X}^*$ is a fixed point of $\mathcal{H}(X)$ \footnote{This result depends on the continuity of $\mathbf{H}(X)$ at $\mathbf{X}^*$, which requires a formal proof. We conjecture that $\mathcal{H}(X)$ is continuous at the points where it is well-defined.}, as shown in Proposition~\ref{proposition:continuity}.

\begin{proposition}\label{proposition:continuity}.
For any sequence $\mathbf{X}^k$ converging to $\mathbf{X}^*$ we will have
$$\mathbf{X}^* = \mathcal{H}(\mathbf{X}^*).$$
\end{proposition}


We can see that if $\mathbf{X}^k$ converges to $\mathbf{X}^*$, then $\mathbf{X}^*$ is a fixed point of $Y = \mathcal{H}(X)$. We define a set $\mathcal{F}$ as $$\mathcal{F} = \{X|\mathbf{X} = \mathcal{H}(\mathbf{X}), g(\mathbf{X})\leq 0\},$$
and provide a necessary condition for the convergence of $\mathbf{X}^k$ in Lemma~\ref{lemma:necessarycondition}.

\begin{lemma}\label{lemma:necessarycondition}
Under the condition that the dual variable $\mathbf{U}^k$ converges to $\bar{\mathbf{U}}$, if $\mathbf{X}^k$ converges to $\mathbf{X}^*$, then $\mathbf{X}^*\in\mathcal{F}$.
\end{lemma}

Lemma~\ref{lemma:necessarycondition} is useful in the sense that, with the knowledge of $\mathbf{D}$, we can restrict our attention to $\mathbf{F}$. On one hand, if we observe that the dual variable converges, we do not need to keep updating since the only possible limiting points are in $\mathcal{F}$; on the other hand, if we find $\mathcal{F}$ is empty, then we can say that the primal variable will not converge and there is no need to spend more effort on updating.

\paragraph{the Fixed Points of $\mathcal{H}(X)$}

It turns out that the fixed points of $Y = \mathcal{H}(X)$ are highly related to the singular value decomposition of $\mathbf{D}$, which we will show in this section.

To warm up, we firstly provide one fixed point in Corollary~\ref{corrollary:one_fixed_point}.
\begin{corollary}\label{corrollary:one_fixed_point}
If $\tilde{\mathbf{X}}$ is the optimal solution of $\underset{\text{rank}(X)\leq K}{\min}\|X-D\|_F$ by truncated SVD, then $\tilde{\mathbf{X}}=\mathcal{H}(\tilde{\mathbf{X}})$.
\end{corollary}

\begin{proof}
Let $\mathbf{D} = U\Sigma V^T$, we can obtain that the optimal solution is $\tilde{\mathbf{X}} = U\left[
       \begin{array}{cc}
         \Sigma_{1:K} & \mathbf{0} \\
         \mathbf{0} & 0 \\
       \end{array}
     \right] V^T$. Then,
\begin{align*}
\alpha \tilde{\mathbf{X}}+(1-\alpha)D &= \alpha U\left[
       \begin{array}{cc}
         \Sigma_{1:K} & \mathbf{0} \\
         \mathbf{0} & 0 \\
       \end{array}
     \right] V^T + (1-\alpha)U\Sigma V^T\\
& = U\left[
       \begin{array}{cc}
         \Sigma_{1:K} & \mathbf{0} \\
         \mathbf{0} & (1-\alpha)\Sigma_{K+1:N} \\
       \end{array}
     \right] V^T,
\end{align*}
which is the singular value decomposition of $\alpha \tilde{\mathbf{X}}+(1-\alpha)\mathbf{D}$. Then the value of $\mathcal{H}(\mathbf{X}^*)$ is $ U\left[
       \begin{array}{cc}
         \Sigma_{1:K} & \mathbf{0} \\
         \mathbf{0} & 0 \\
       \end{array}
     \right] V^T$.

The proof is completed.
\end{proof}

We note that even though $\tilde{\mathbf{X}}$ is a fixed point of $\mathbf{Y} = \mathcal{H}(\mathbf{X})$, it may not be a stationary point of $\mathbf{X}$ update because $\tilde{\mathbf{X}}$ may not satisfy $g(\mathbf{X})\leq 0$.

Next, we characterize all possible fixed points of $\mathcal{H}(X)$ in Corollary~\ref{corollary:all_fixed_points}.

\begin{corollary}\label{corollary:all_fixed_points}
Suppose $\mathbf{D}$ admits the singular value decomposition of $U \Sigma V^T$ with $\sigma_1\geq \sigma_2\geq ...\geq \sigma_n\geq 0$, and $\mathcal{I} = \{i_1,i_2,...,i_K\}$ is a subset of $\{1,2,...,n\}$ with size $K$. If for any $i\in \mathcal{I}$ and $j\notin \mathcal{I}$ we can have $\sigma_{i}\geq (1-\alpha)\sigma_{j}$, then $\tilde{\mathbf{X}} = \sum_{i\in\mathcal{I}}\sigma_iu_iv_i^T$ is a fixed point of $Y=\mathcal{H}(\mathbf{X})$. On the other hand, if $\tilde{\mathbf{X}}$ is a fixed point, then there must exists such a subset $\mathcal{I}$, such that $\tilde{\mathbf{X}} = \sum_{i\in\mathcal{I}}\sigma_iu_iv_i^T$.
\end{corollary}
\begin{proof}
Suppose $\tilde{\mathbf{X}}$ is a fixed point of $\mathcal{H}(\mathbf{X})$ and $\alpha \tilde{\mathbf{X}}+(1-\alpha)D$ admits the singular value decomposition of $\mathbf{U}\Sigma \mathbf{V}^T$. Then $\tilde{\mathbf{X}} = \mathbf{U}\left[
                                                                                               \begin{array}{cc}
                                                                                                 \Sigma_{1:K} & 0 \\
                                                                                                 0 & 0 \\
                                                                                               \end{array}
                                                                                             \right]\mathbf{V}^T$.
By substituting into $(1-\alpha)D+\alpha\tilde{\mathbf{X}} = \mathbf{U}\Sigma \mathbf{V}^T$ we can have
$$(1-\alpha)D + \alpha \mathbf{U}\left[\begin{array}{cc}
                                                                                                 \Sigma_{1:K} & 0 \\
                                                                                                 0 & 0 \\
                                                                                               \end{array}
                                                                                             \right]\mathbf{V}^T =  \mathbf{U}\Sigma \mathbf{V}^T.$$
Thus $$D = \mathbf{U}\left[\begin{array}{cc}
                                                                                                 \Sigma_{1:K} & 0 \\
                                                                                                 0 & \frac{\Sigma_{K+1:n}}{1-\alpha} \\
                                                                                               \end{array}
                                                                                             \right]\mathbf{V}^T.$$
However, to have the standard SVD of $D$, we need to permutate the diagonal elements of $\left[\begin{array}{cc}
                                                                                                 \Sigma_K & 0 \\
                                                                                                 0 & \frac{\Sigma_{K+1:n}}{1-\alpha} \\
                                                                                               \end{array}
                                                                                             \right]$ to make them decreasing and the corresponding columns of $\mathbf{U},\mathbf{V}$.

However, the diagonal elements of $\left[\begin{array}{cc}
                                                                                                 \Sigma_{1:K} & 0 \\
                                                                                                 0 & \Sigma_{K+1:n} \\
                                                                                               \end{array}
                                                                                             \right]$ are decreasing, which establishes the results in this corollary.
\end{proof}

We provide some remarks regarding this corollary.
\begin{itemize}
\item By Corollary~\ref{corollary:all_fixed_points}, it is easy to see that Corollary~\ref{corrollary:one_fixed_point} is true.
\item The number of fixed points is upper bounded by $n\choose{K}$ because we need to choose $K$ singular values to form one fixed point of $\mathcal{C}(X)$.
\item A larger value of $\alpha$ (meaning a larger $\rho$) will make the condition $\sigma_{i}\geq (1-\alpha)\sigma_{j}$ more difficult to satisfy and the possible fixed points will be fewer.
\end{itemize}

\section{Simulations}\label{sec:simulation}
In this section, we conduct some simulations to show the performance of the proposed algorithm. Our purpose is to investigate (i) the performance comparison with other algorithms under different scenarios, (ii) the impact of values of $\rho$, and (iii) the performance on a real application.

\subsection{Non-negative low-rank approximation}

Firstly, we consider a special case with the convex constraints being $\mathbf{X}_{i,j}\geq 0, \forall i,j$, \textit{i.e.}, the non-negative low-rank approximation problem.

Other than the algorithm we propose, there are two alternatives to solve this problem: Alternating Direction Projection (\textbf{ADP}) \cite{chu2003structured}, and Non-negative Matrix Factorization (\textbf{NMF}) \cite{markovsky2008structured}. We briefly discuss them here.

In each iteration of \textbf{ADP}, we first make a projection onto the rank-K set and then another projection onto the convex set $\{\mathbf{X}|g(\mathbf{X})\leq 0\}$, both of which can be solved efficiently. One advantage of \textbf{ADP} is that the algorithm is guaranteed to converge, but this is also its disadvantage, because it is easily to be trapped by a local optimal. In \textbf{NMF}, we firstly equivalently transform the rank-K constraints to the fact that $\mathbf{X} = \mathbf{A}\mathbf{B}, \mathbf{A}\in \mathbb{R}^{M\times K}, \mathbf{B}\in \mathbb{R}^{K\times N}$ and impose constraints $\mathbf{A}_{i,j}\geq 0, \mathbf{B}_{i,j}\geq 0$ to ensure that $\mathbf{X}_{i,j}\geq 0$, which is obviously more strict. Then we apply the ADMM algorithm propose in \cite{zhang2010alternating} to obtain the factorization.

In the experiments, the original data $\hat{\mathbf{X}}\in \mathbf{R}^{100\times 80}$ is randomly generate and $\rho = 5$ for the algorithm we propose. The objective values of different algorithms in different scenarios (different rank constraints) are shown in from Fig~\ref{fig:NLRA1} to Fig~\ref{fig:NLRA3}.

\begin{figure}
\begin{minipage}{0.31\columnwidth}
\includegraphics[width=0.9\columnwidth]{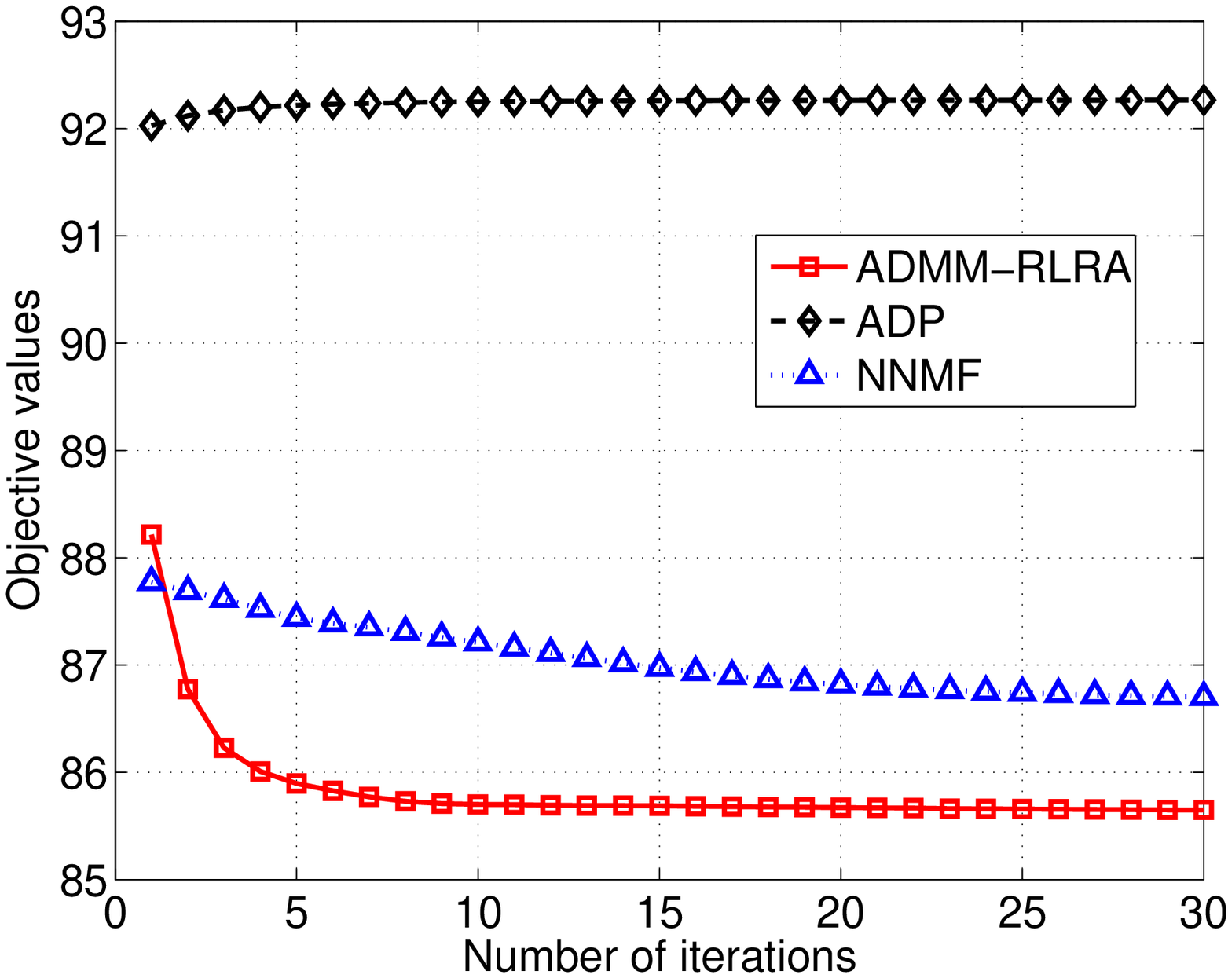}
\caption{Non-negative low-rank approximation with $K = 3$}\label{fig:NLRA1}
\end{minipage}
\begin{minipage}{0.31\columnwidth}
\includegraphics[width=0.9\columnwidth]{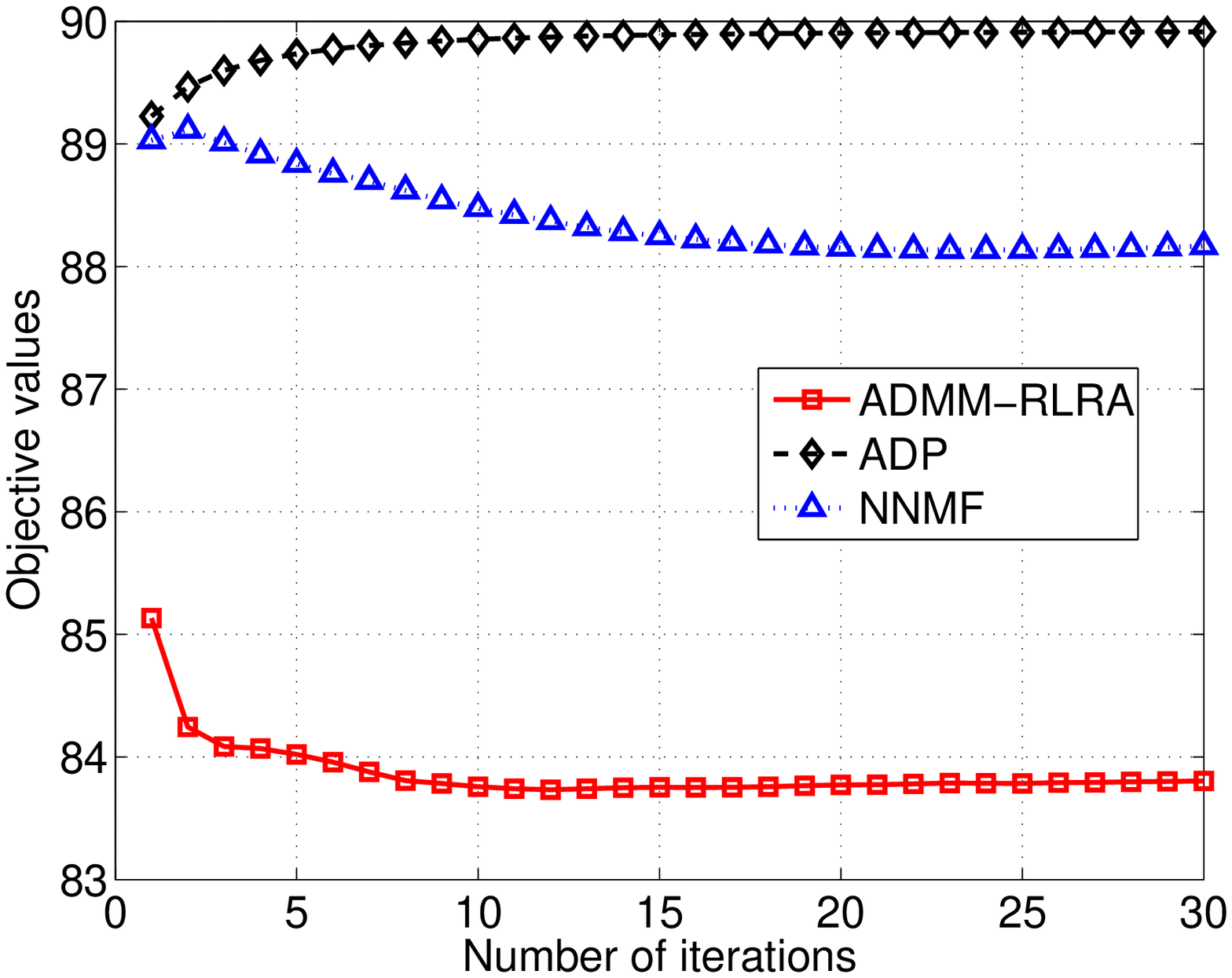}
\caption{Non-negative low-rank approximation with $K = 6$}\label{fig:NLRA2}
\end{minipage}
\begin{minipage}{0.31\columnwidth}
\includegraphics[width=0.9\columnwidth]{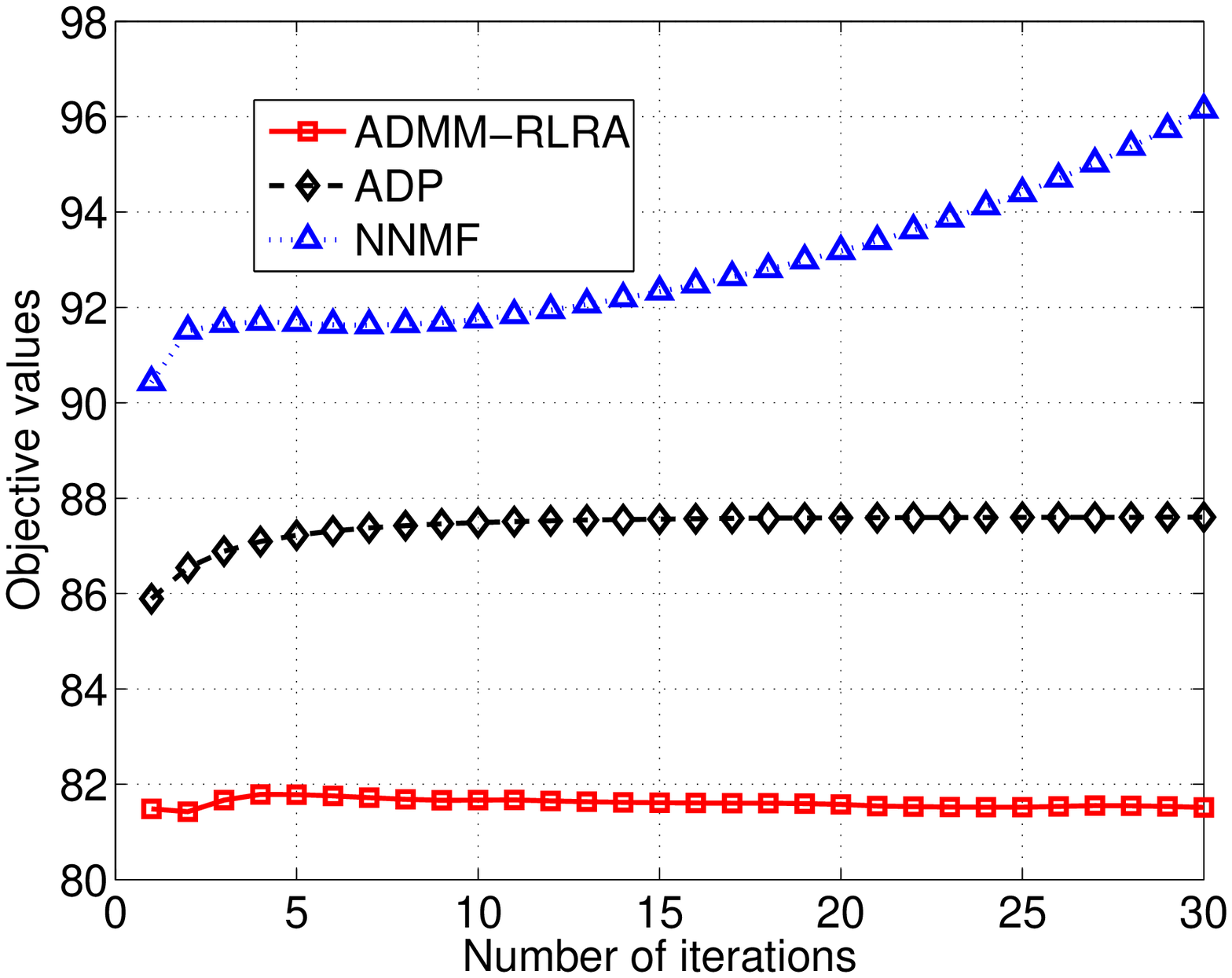}
\caption{Non-negative low-rank approximation with $K =10$}\label{fig:NLRA3}
\end{minipage}
\end{figure}

It can be seen that the performance of our proposed algorithm is always the best in three cases. \textbf{NMF} fails to converge when $K = 10$. Another observation is that the objective value of \textbf{AD} will remain the same after several iterations, the reason of which is that the optimization variable will not change once it is feasible. The reason that our algorithm outperforms \textbf{NMF} is that \textbf{NMF} requires $\mathbf{A}$ and $\mathbf{B}$ are both non-negative, which is more restrictive than the original problem.

\subsubsection{The impact of $\rho$}\label{sec:ExpRho}

We have little theoretical results on how the parameter $\rho$ will impact the performance and how to choose a proper value of $\rho$ for different problems. We provide some simulation results regarding to the convergence and performance here. In this simulation, we fix the data matrix $\hat{\mathbf{X}}$, keep $K =5$ and examine the residual value $\|\mathbf{X}-\mathbf{Y}\|_F$ and objective value for $\rho = 1,5,9,15$. The result is shown in Fig~\ref{fig:RhoResidual} \ref{fig:RhoObj}.

\begin{figure}
\begin{minipage}{0.45\columnwidth}
\includegraphics[width=0.9\columnwidth]{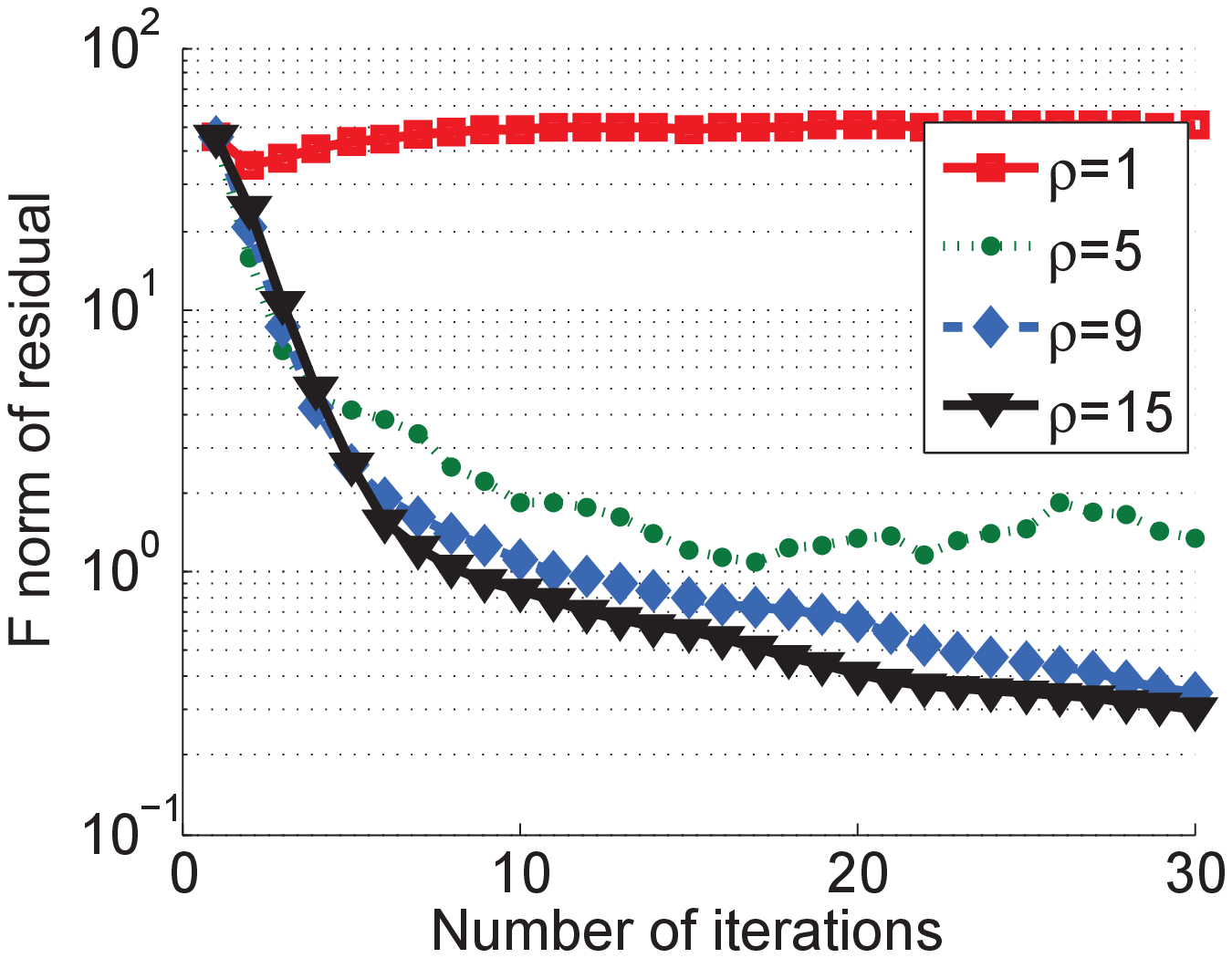}\\
\caption{Values of residual with different $\rho$}\label{fig:RhoResidual}
\end{minipage}
\begin{minipage}{0.45\columnwidth}
\includegraphics[width =0.9\columnwidth]{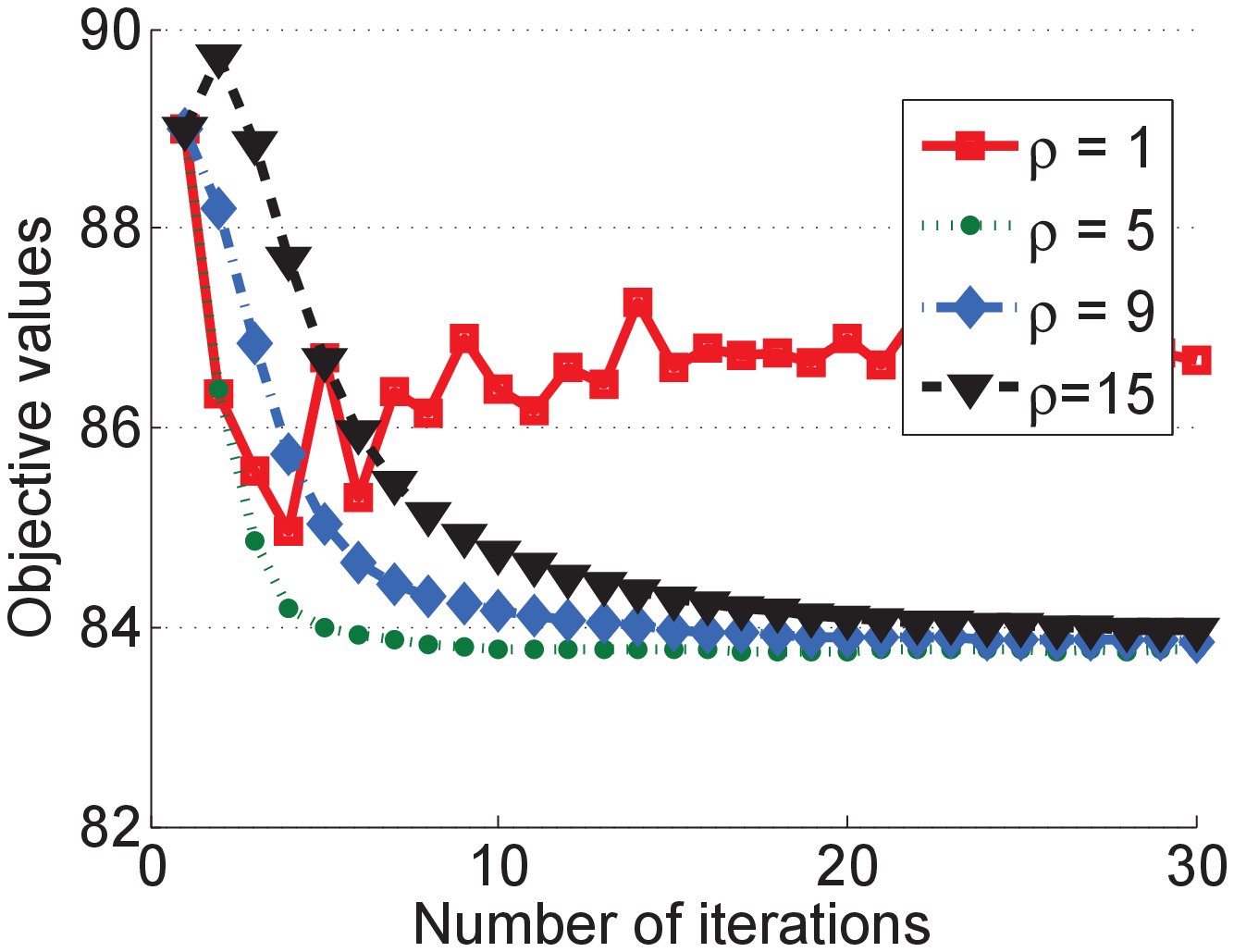}\\
\caption{Values of objective with different $\rho$}\label{fig:RhoObj}
\end{minipage}
\end{figure}

From Fig~\ref{fig:RhoResidual}, we can see that the residual value converges to $0$ when $\rho = 5,9,15$ and converges faster with a larger $\rho$, but fails to converge when $\rho=1$; on the other hand, as shown in Fig~\ref{fig:RhoObj}, for the scenarios when the algorithm converges, the objective value decreases more slowly if $\rho$ is larger. As we can see, a larger $\rho$ will make the algorithm more easily to converge, but will result in a worse performance, indicating that there is a underlying tradeoff, which we leave for future work.

\subsection{Image denoising by partial prior knowledge}

Next, we consider an application of image denoising. More precisely, we have a noisy image $\hat{\mathbf{X}}$ and two kinds of prior knowledge: (a) the original image is of low rank and the rank upper bound is $K$; (b) the original values of some pixels. And we want to obtain a clear image with less noise.

Singular value decomposition is shown useful in image denoising \cite{rajwade2013image} \cite{hou2003adaptive}. But direct truncated SVD cannot use the second kind of prior knowledge in the scenario we consider. Here we formulate an optimization problem in the form of \textbf{RLRA} with the convex constraint $g(\mathbf{X})\leq 0$ being $\mathbf{X}_{i,j} = x_{i,j}, (i,j)\in \mathcal{P}$, where $\mathcal{P}$ is the set of pixels with known values. We do the experiment on MIT logo (the rank of the original image is 5) with $\rho = 5$ and the values of $5\%$ of the pixels are known as a prior. The truncated SVD algorithm (\textbf{TSVD}) is used as a comparison. The original image, noised image, denoised image by \textbf{ADMM-RLRA} and denoised image by \textbf{TSVD} are shown from Fig~\ref{fig:mit} to Fig~\ref{fig:TSVDmit}. The image qualities measured by PSNR and SNR are shown in Table~\ref{table:imagequality}.

\begin{table}[h]
\centering
\caption{Image quality, PSNR, SNR}
\label{table:imagequality}
\begin{tabular}{|l|l|l|}
\hline
            & PSNR   & SNR    \\ \hline
Noisy Image & -3.504 & -4.895 \\ \hline
By \textbf{TSVD}     & 8.675  & 7.284  \\ \hline
By \textbf{RLRA}     & $\mathbf{12.895}$ & $\mathbf{11.504}$ \\ \hline
\end{tabular}
\end{table}

The simulation results shows that the second kind of prior knowledge is useful and \textbf{ADMM-RLRA} can use that to almost double the PSNR and improve the image quality significantly, which verifies the effectiveness of our algorithm.
\begin{figure}
\begin{minipage}{0.45\columnwidth}
\includegraphics[width=0.9\columnwidth]{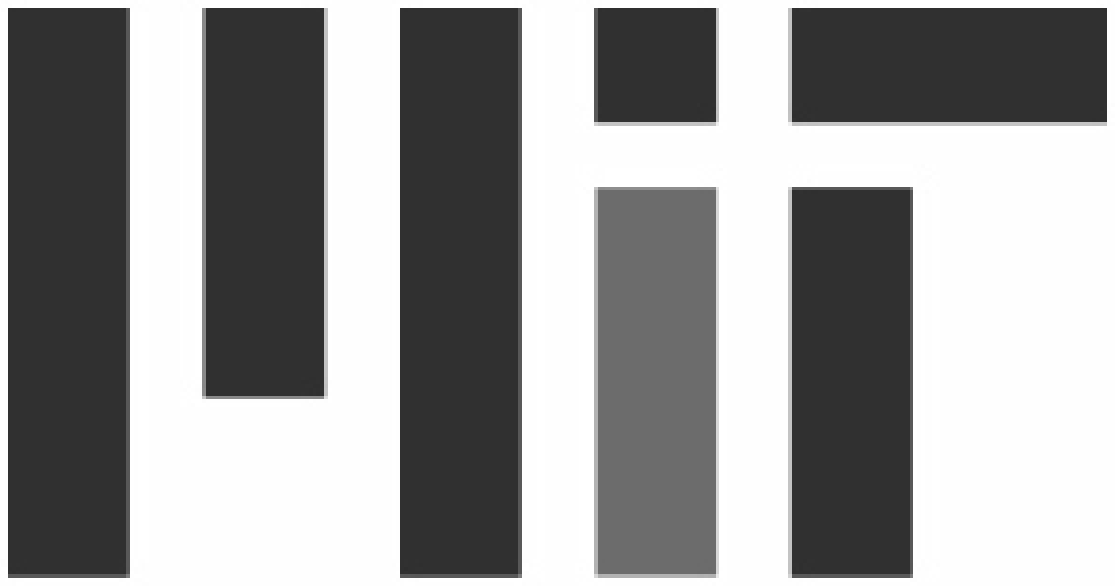}\\
\caption{Original image: MIT logo}\label{fig:mit}
\end{minipage}
\begin{minipage}{0.45\columnwidth}
\includegraphics[width =0.9\columnwidth]{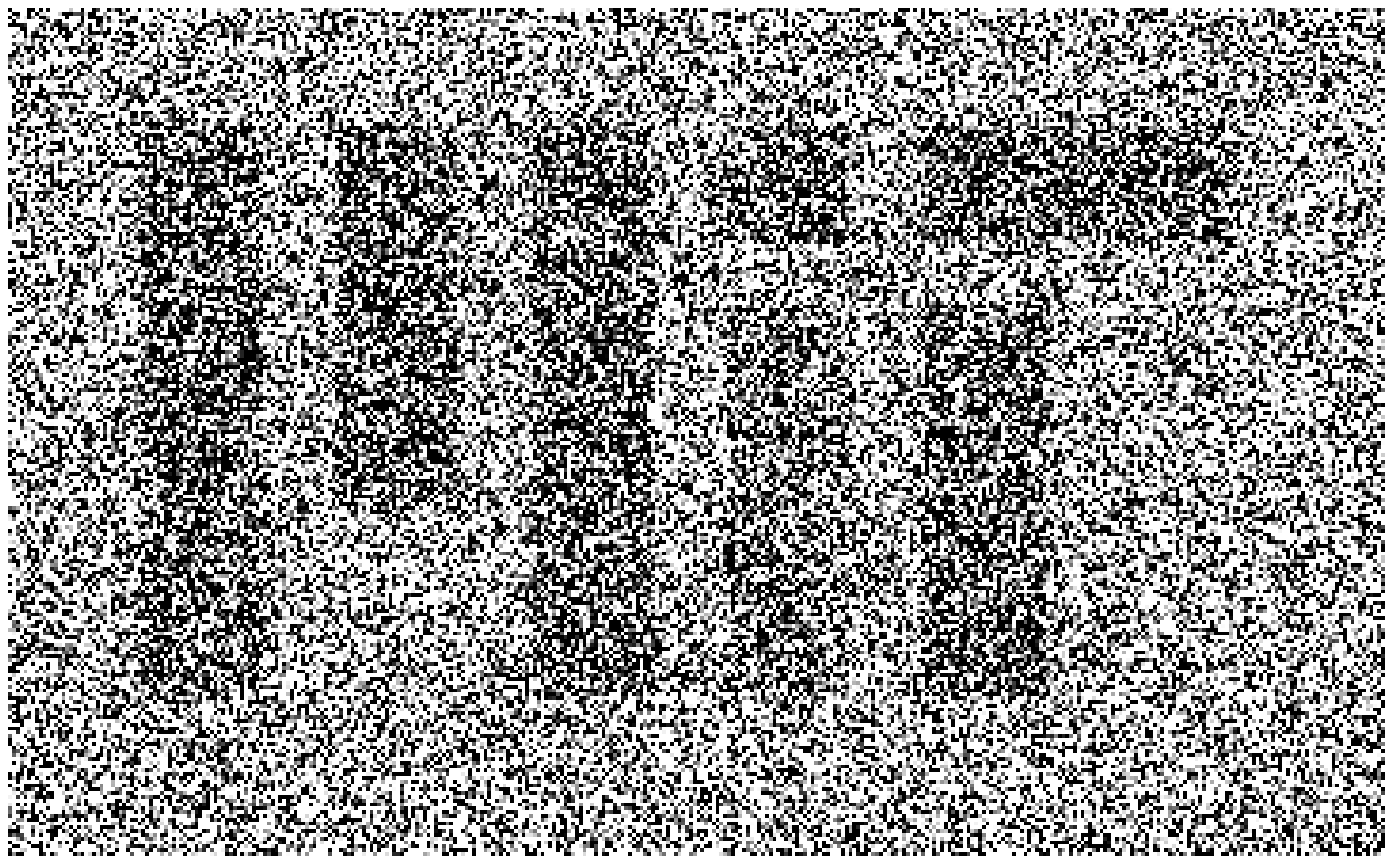}\\
\caption{Noisy image with Gaussian noise}\label{fig:noisedmit}
\end{minipage}
\end{figure}

\begin{figure}
\begin{minipage}{0.45\columnwidth}
\includegraphics[width=0.9\columnwidth]{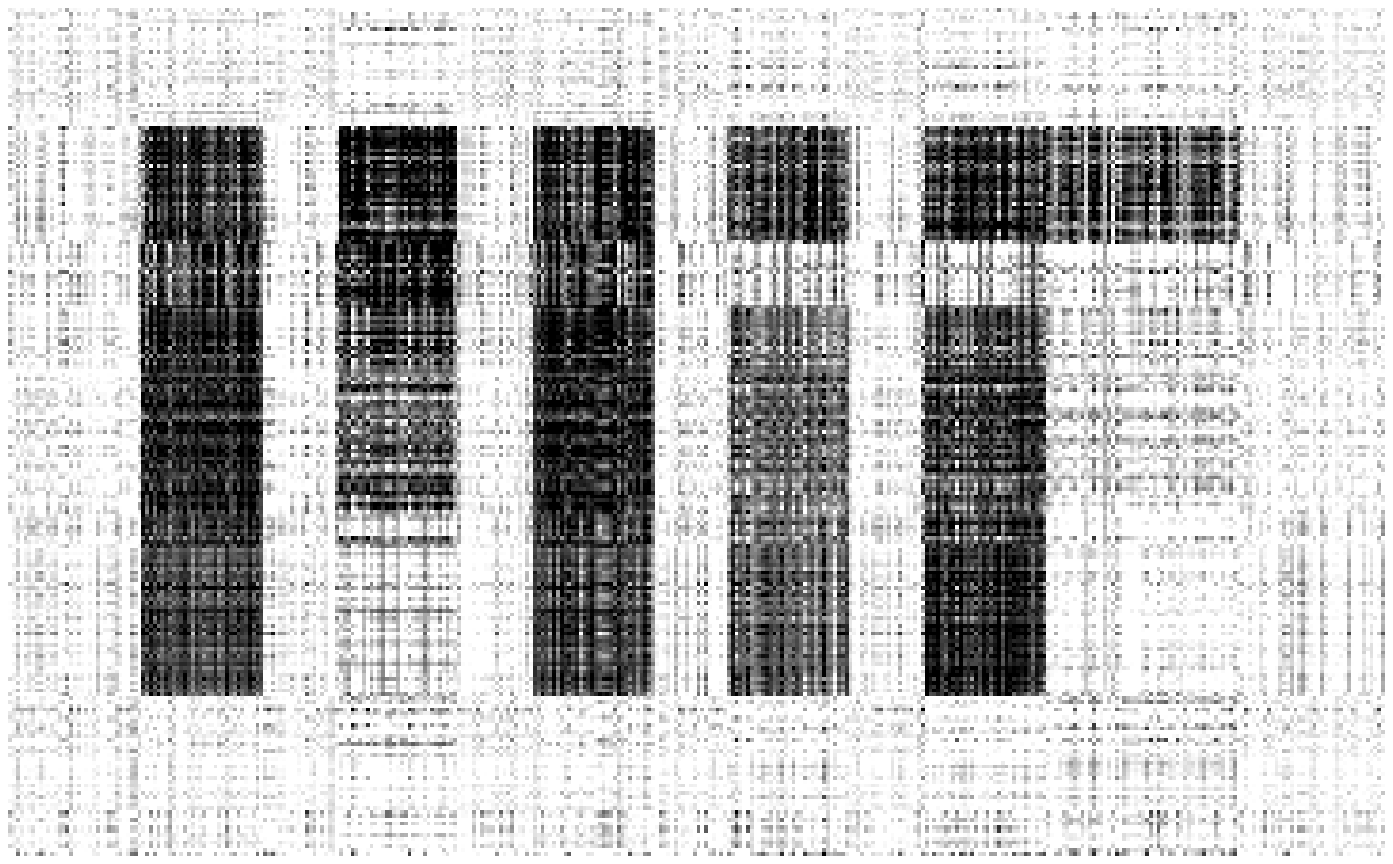}\\
\caption{Recovered image by \textbf{RLRA}}\label{fig:RLRAmit}
\end{minipage}
\begin{minipage}{0.45\columnwidth}
\includegraphics[width =0.9\columnwidth]{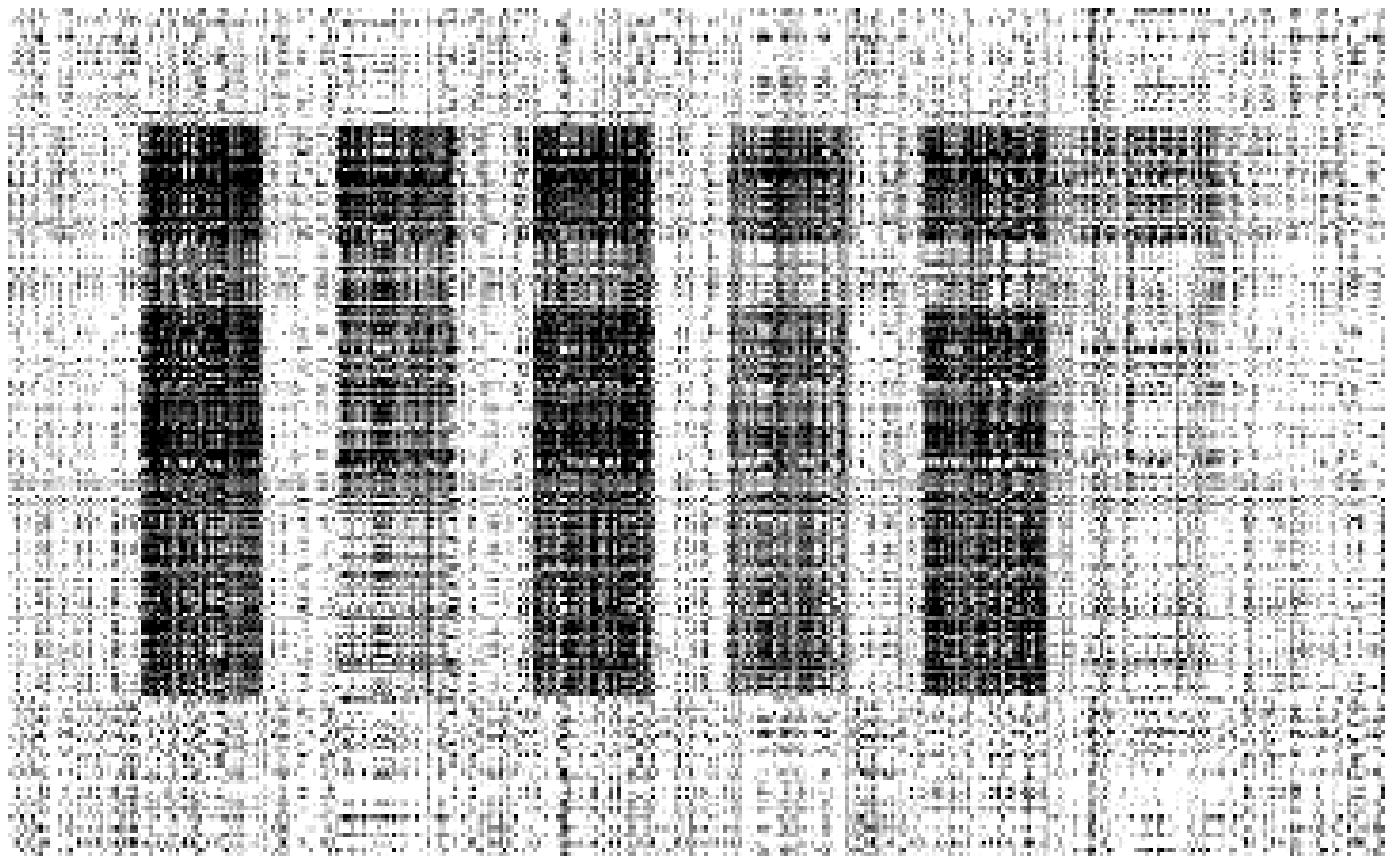}\\
\caption{Recovered image by \textbf{TSVD}}\label{fig:TSVDmit}
\end{minipage}
\end{figure}

\section{Conclusion and Future Work}\label{sec:conclusion}

This paper proposes to use ADMM algorithm to solve the restricted low-rank approximation problem. We show that all subproblems in ADMM can be solved efficiently. And more interestingly, under the condition that the dual variable converges, the update of primal variable is carried out by a function and the possible limiting points are its stationary points,.
We provide some preliminary theoretical results based on this understanding and show its performance by experiments.

The theoretical part in this paper is limited because most of them is based on the assumption that the dual variable converges, which is not guaranteed. The most important future work is to study under what condition the dual variable will converge. From Corollary~\ref{corollary:all_fixed_points} and experiments in Section~\ref{sec:ExpRho}, we can see that the parameter $\rho$ plays an important role and an interesting question is to ask how the value of $\rho$ will affect the convergence of this algorithm and how to choose a proper $\rho$.

There can be several different ways to reformulate the original problem and apply ADMM. As we mention in Section~\ref{sec:otherorder}, we can change the update order. And also, we can let $\mathcal{I}(X)$ be the indicator function for the convex constraints and $\mathcal{J}(Y)$ be the indicator function for the rank constraint. It is interesting to study how the different approaches will affect the solution theoretically and empirically. 
\bibliographystyle{abbrv}
\bibliography{ref}

\end{document}